\documentclass[12pt]{article}
\usepackage{srcltx}
\usepackage{amsfonts,amssymb,mathrsfs,amsmath,cmtiup, enumerate}
\usepackage{amsthm}

\newtheorem{theorem}{Theorem}[section]
\newtheorem{lemma}[theorem]{Lemma}
\newtheorem{proposition}[theorem]{Proposition}
\newtheorem{corollary}[theorem]{Corollary}

\theoremstyle{definition}
\newtheorem*{definition}
{Definition}
\newtheorem{remark}[theorem]{Remark}
\newtheorem*{Index Convention}{Index Convention}
\newtheorem*{notation}{Notation}

\def\keywords#1{\par\medskip
\noindent\textbf{Keywords.} #1}

\def\subjclass#1{{\renewcommand{\thefootnote}{}%
\footnote{\emph{Mathematics Subject Classification (2010):} #1}}}

\begin{document}
\let\le=\leqslant
\let\ge=\geqslant
\let\leq=\leqslant
\let\geq=\geqslant
\newcommand{\e}{\varepsilon }
\newcommand{\f}{\varphi }
\newcommand{ \g}{\gamma}
\newcommand{\F}{{\Bbb F}}
\newcommand{\N}{{\Bbb N}}
\newcommand{\Z}{{\Bbb Z}}
\newcommand{\Q}{{\Bbb Q}}
\newcommand{\C}{{\Bbb C}}
\newcommand{\R}{\Rightarrow }
\newcommand{\W}{\Omega }
\newcommand{\w}{\omega }
\newcommand{\s}{\sigma }
\newcommand{\hs}{\hskip0.2ex }
\newcommand{\ep}{\makebox[1em]{}\nobreak\hfill $\square$\vskip2ex }
\newcommand{\Lr}{\Leftrightarrow }

\title{Nilpotency of Lie type algebras  with metacyclic Frobenius groups of automorphisms}

\markright{}

\author{{\sc N.\,Yu.~Makarenko}\\ \small Sobolev Institute of Mathematics, Novosibirsk, 630\,090,
Russia }

\author{
{N.\,Yu.~Makarenko\footnote{The work is supported by  Russian
Science Foundation, project 21-11-00286 https://rscf.ru/project/21-11-00286/}}\\
\small  Sobolev Institute of Mathematics, Novosibirsk, 630\,090,
Russia
\\[-1ex] \small  natalia\_makarenko@yahoo.fr}
\date{}
\maketitle


\begin{abstract} An algebra $L$ over a field $\Bbb
K$, in which product is denoted by $[\,,\,]$, is called a \textit{
Lie type algebra}  if for all elements $a,b,c\in L$ there exist
$\alpha, \beta\in \Bbb K$  (depending on $a,b,c$) such that
$\alpha\neq 0$ and $[[a,b],c]=\alpha [a,[b,c]]+\beta[[a,c],b]$.
 Suppose that a Lie type algebra $L$  admits a Frobenius group of automorphisms $FH$ with cyclic
kernel $F$ of order $n$ and complement $H$ such that the
fixed-point subalgebra of $F$ is trivial and the fixed-point
subalgebra of $H$ is nilpotent of class $c$. If either the ground
field $\Bbb K$ contains a primitive $n$th  root of unity, or
$\alpha, \beta $ are constant in $[[a,b],c]=\alpha
[a,[b,c]]+\beta[[a,c],b]$ (i.e. do not depend on the choice of
$a,b,c$), then $L$ is nilpotent and the nilpotency class of $L$ is
bounded in terms of $|H|$ and $c$. The result extends the known
theorem of Khukhro, Makarenko and Shumyatsky on Lie algebras with
metacyclic Frobenius group of automorphisms.
\end{abstract}


\section{Introduction}

\vskip1ex
  A \textit{Lie type algebra} is by definition an algebra $L$ over a field
$\Bbb F$ with product $[\,,\,]$ satisfying the following property:
for all elements $a, b, c \in L$ there exist $\alpha, \beta \in
\Bbb F$ such that $\alpha\ne 0$ and
\begin{equation}\label{tozh}
[[a,b],c]=\alpha [a,[b,c]]+\beta[[a,c],b].\end{equation} Note that
in general $\alpha, \beta$ depend on elements $a,b,c\in L$;  they
can be viewed as functions $\alpha, \beta: L\times L \times L
\rightarrow \Bbb F$.

\vskip1ex If a group $G$ acts on an algebra $L$, we denote by
$$C_L(G)=\{l\in L \,\,\,\mid \,\,\, l^{\varphi}=l \mathrm{\,\,\,
for\, all\,\, } \varphi\in G\}$$  the fixed-point subalgebra of
$G$.

\vskip1ex
 By  a theorem of Khukhro, Makarenko and  Shumyatsky~\cite{khu-ma-shu} if  a Lie algebra over a field admits  a Frobenius group of automorphisms $FH$ with cyclic
kernel $F$ and complement $H$ such that the fixed-point subalgebra
of $F$ is trivial and the fixed-point subalgebra of $H$ is
nilpotent of class $c$, then $L$ is nilpotent  and the nilpotency
class of $L$ is bounded in terms of $|H|$ and $c$.

\vskip1ex The aim of the present paper is to extent this result to
the class of Lie type algebras, which includes in particular Lie
algebras, associative algebras and Leibnitz algebras.

 \vskip1ex
 \begin{theorem}\label{th-1}
 Suppose that a Lie type  algebra $L$ \/$($of  possibly infinite dimension\/$)$ over an arbitrary field $\Bbb K$ admits a Frobenius group of automorphisms $FH$ with cyclic
kernel $F$ of order $n$ and complement $H$ of order $q$ such that
the fixed-point subalgebra
 of $H$ is nilpotent of class $c$ and the fixed-point subalgebra
of $F$ is trivial. Assume also that  either $\Bbb K$ contains a
primitive $n$th root of unity, or $\alpha, \beta$ in property
(\ref{tozh}) are constants and do not depend on the choice of
$a,b,c$. Then $L$ is nilpotent and the nilpotency class of $L$ is
bounded in terms of $q$ and $c$.
\end{theorem}

\vskip1ex

 The proof of the theorem  is heavily based on the arguments of~\cite{khu-ma-shu}. Since the demonstrations in~\cite{khu-ma-shu}
do not use structure theory and all calculations are founded on
combinatorial reasoning in a graded Lie algebra, most of the
lemmas  can be easily adapted  to a graded Lie
 type algebra  by  replacing the Jacobi identity by  the propery~(\ref{tozh}). We give them without proof. However, unlike Lie
algebras, Lie type algebras lack the anticommutativity identity,
therefore some works are needed to overcome this difficulty.

\vskip1ex

We now state a straightforward corollary of Theorem \ref{th-1}.

\vskip1ex
 A \textit{\/$($right\/$)$ Leibniz
algebra} or  \textit{Loday algebra} is an algebra $L$ over a field
 satisfying the Leibniz
identity $$[[a,b],c] = [a,[b,c]]+ [[a,c],b]$$ for all $a,b,c\in
L$.

\vskip1ex

\begin{corollary}\label{th1}
 Suppose that a Leibniz   algebra $L$ \/$($of  possibly infinite dimension\/$)$ over an arbitrary field $\Bbb K$ admits a Frobenius group of automorphisms $FH$ with cyclic
kernel $F$ and complement $H$ of order $q$ such that the
fixed-point subalgebra
 of $H$ is nilpotent of class $c$ and the fixed-point subalgebra
of $F$ is trivial.  Then $L$ is nilpotent and the nilpotency class
of $L$ is bounded in terms of $q$ and $c$.
\end{corollary}

\vskip1ex  Some preliminary definitions and facts are given in
\S\,2. In \S\,3 we prove some auxiliary result --- an analogue for
 Lie type algebras of Shalev--Kreknin theorem \cite{kr, shalev}
on graded Lie algebras with small number of non-trivial components
(Proposition \ref{small-comp}). This result is used in \S\,5 to
prove the solvability of bounded derived length of $L$. But before
that, we perform the reduction to graded algebras in \S\,4.
Finally in \S\,6 we employ the induction on derived length to
prove the nilpotency of bounded class of $L$.

\section{Preliminaries}

Let $L$ be a  Lie type algebra. If $M,\, N$ are subspaces of
   $L$, then $[M,N]$
denotes the subspace, generated  by all the products $[m,n]$ for
$m\in M$, $n\in N$. In view of (\ref{tozh}),  if $M$ and $N$ are
two-side ideals, then $[M,N]+[N,M]$ is also a two-side ideal; if
$H$ is a (sub)algebra, then $[H,H]$ is  two-side ideal of $H$ and,
in particular, its subalgebra.
 The  subalgebra  generated by  subspaces~$U_1,U_2,\ldots, U_k$ is
denoted by $\left<U_1,U_2,\ldots, U_k\right>$, and the two-side
ideal generated by~$U_1,U_2,\ldots, U_k$ is denoted by ${}_{\rm
id}\!\left<U_1, U_2,\ldots, U_k\right>$.

\vskip1ex
 A simple product  $[a_1,a_2,a_3,\ldots, a_s]$ is by
definition the left-normalized product
$$[...[[a_1,a_2],a_3],\ldots, a_s].$$ The analogous notation is also
used for subspaces
$$[A_1,A_2,A_3,\ldots, A_s]=[...[[A_1,A_2],A_3],\ldots, A_s].$$

\vskip1ex Since
\begin{equation}\label{tozh2} [a,[b,c]]=\frac{1}{\alpha}\,[[a,b],c]-\frac{\beta}{\alpha}\,[[a,c],b]\end{equation}
for all $a,b,c\in L$,
 any (complex) product in  elements in $L$  can be expressed as a linear combination of simple products of
the same length in the same elements.  Also, it follows that the
(two-sided) ideal in  $L$ generated by a  subspace $S$ is the
subspace generated by all the  simple products
$[x_{i_1},y_j,x_{i_2},\ldots, x_{i_t}]$ and $[y_j,x_{i_1},
x_{i_2},\ldots x_{i_t}]$, where $t\in \Bbb N$ and $x_{i_k}\in L,
y_j\in S$. In particular, if $L$ is generated by a  subspace $M$,
then  its space is spanned  by simple products in elements of~$M$.

 \vskip1ex
 The derived series of an algebra $L$ is defined as
$$L^{(0)}=L, \; \; \; \, \, \, \, \, \,
L^{(i+1)}=[L^{(i)},L^{(i)}].$$ Then $L$ is solvable of derived
length at most $n$ if $L^{(n)}=0$.

 \vskip1ex

Terms of the lower central series of $L$ are defined as $\gamma
_1(L)=L;$ \ $\gamma _{k+1}(L)=[\gamma _k(L),\,L]$. Then $L$ is
nilpotent of class at most $c$ if $\gamma _{c+1}(L)=0$.

\vskip1ex

 We will need the following algebra analog of P.~Hall's theorem
\cite{hall}, which will helps us in proving nilpotency of a
solvable Lie type algebra.

\vskip1ex
\begin{lemma}\label{chao} Let $K$ be an ideal of a Lie
type algebra $L$. If $\gamma _{c+1}(L)\subseteq [K,K]$ and $\gamma
_{k+1}(K)=0$ then $$\gamma _{c{k+1 \choose 2} -{k \choose
2}+1}(L)=0.$$
\end{lemma}
 \begin{proof} The proof can be easily obtained from the proof of P.~Hall's theorem
 \cite{hall}  by appropriate simplifications in our ``linear'' case.
\end{proof}

 \vskip1ex

  An algebra $L$ over a field
is \textit{ $({\Bbb Z}/n{\Bbb Z})$-graded} if
$$L=\bigoplus_{i=0}^{n-1}L_i\qquad \text{ and }\qquad [L_i,L_j]\subseteq L_{i+j\,({\rm mod}\,n)},$$
where  $L_i$ are subspaces  of~$L$. Elements of $L_i$ are referred
to as \textit{homogeneous} and the subspaces $L_i$  are called
\textit{homogeneous components} or  \textit{grading components}.
In particular, $L_0$ is called the zero component.

 \vskip1ex

An additive subgroup $H$ of $L$ is called \textit{homogeneous} if
$H=\bigoplus_i (H\cap L_i)$ and we set $H_i=H\cap L_i$. Obviously,
any subalgebra or an ideal generated by homogeneous subspaces is
 homogeneous. A homogeneous subalgebra  can be regarded as a
$({\Bbb Z} /n{\Bbb Z} )$-graded algebra with the induced grading.
It follows that the terms of the derived series and of the lower
central  series of $L$, the ideals $L^{(k)}$ and  $\gamma
_{k}(L)$, are also $({\Bbb Z} /n{\Bbb Z} )$-graded algebras with
induced grading $$L^{(k)}_i=L^{(k)}\cap L_i,\,\,\,\,\,\gamma
_{k}(L)_i=\gamma _{k}(L)\cap L_i$$ and
$$L^{(k+1)}_i=\sum_{u+v\equiv i\,({\rm mod\,}n)
}[L^{(k)}_{u},\,L^{(k)}_{v}]$$ $$\gamma
_{k+1}(L)_i=\sum_{u+v\equiv i\,({\rm mod\,}n ) }\Big([\gamma
_{k}(L)_u,\,L_{v}]+[L_v,\,\gamma _{k}(L)_u]\Big)=\sum_{u+v\equiv
i\,({\rm mod\,}n ) }[\gamma _{k}(L)_u,\,L_{v}].$$ The last
equality in the above formula is due to \eqref{tozh2}.

\vskip1ex
\begin{Index Convention}
Henceforth a small Latin letter with an index $i\in \Bbb Z/n\Bbb
Z$ will denote a homogeneous element in the grading component
$L_i$, with the index only indicating which component this element
belongs to: $x_i\in L_i$. We will not be using numbering indices
for elements of the $L_i$, so that different elements can be
denoted by the same symbol when it only matters which component
the elements belong to. For example, $x_{i}$ and $x_{i}$ can be
different elements of $L_{i}$, so that $[x_{i},\, x_{i}]$ can be a
non-zero element of $L_{2i}$.
\end{Index Convention}

 \vskip1ex
We use abbreviation, say, ``$(m,n,\dots )$-bounded'' for ``bounded
above in terms of  $m, n,\dots$''.

\section
{$(\Z/n\Z)$-graded Lie type algebra with small number of
non-trivial homogeneous components}

By Kreknin's theorem \cite{kr} a {$(\Z/n\Z)$-graded Lie algebra
$$L=\bigoplus_{i=0}^{n-1}L_i,$$ where
$[L_i,L_j]\subseteq L_{i+j\,({\rm mod}\,n)}$ and $L_0=0$, is
solvable of derived length at most $2^n-2$.

\vskip1ex In \cite{shalev}, in the frame of studying the finite
groups of bounded rank with automorphisms, Shalev noticed an
interesting fact: if  among $L_i$
  there are only $d\leq n-1$ non-trivial
components, then the derived length does not depend on $n$, but
only on $d$. We will need an analog of this result for Lie type
algebras.

\begin{proposition}\label{small-comp} Let $$L=\bigoplus_{i=0}^{n-1}L_i,$$  be a $(\Z/n\Z)$-graded Lie type algebra.
If $L_0=0$ and among $L_i$ there are only $d\leq n-1$ non-trivial
components, then $L$ is solvable of $d$-bounded derived length.
\end{proposition}

\begin{proof} Let $\Omega =\{ w_1,\ldots ,w_ {d}\}$ be the set of
all indices $i$ such that $L_i\ne 0$. We  assume that
$0<w_1<w_2<\ldots<w_{d-1}<w_d<n$. We use the same arguments as in
the proof of Kreknin's theorem  given in  \cite[Theorem
4.3.1]{kh-book} replacing everywhere $i$ by $w_i$ and the Jacobi
identity  by (\ref{tozh}). The assertion follows from the two
following inclusions
\begin{equation}\label{f-small-comp-1} L^{(2^{k-1})}\cap L_{w_k}
\subseteq \langle L_{w_{k+1}}, L_{w_{k+2}},\ldots,
L_{w_{d}}\rangle \end{equation}
\begin{equation}\label{f-small-comp-2}
L^{(2^{k}-1)}\subseteq \langle L_{w_{k+1}}, L_{w_{k+2}},\ldots,
L_{w_{d}}\rangle,\end{equation} which are proved simultaneously by
induction on $k$.

\vskip1ex We will also need the following elementary Lemma.

\begin{lemma} [{\cite[Lemma 4.3.5]{kh-book}}]\label{ntl} If $i + j \equiv k (\mathrm{mod}\, n)$ for $1 \leq i, j \leq n - 1$, then the numbers $i$ and $j$ are both greater
than $k$ or less than $k$.
\end{lemma}

For  $k=1$ the inclusions (\ref{f-small-comp-1}) and
(\ref{f-small-comp-2}) take the forms
$$L^{(1)}\cap L_{w_1}\subseteq \langle L_{w_2}, \ldots,
L_{w_{d}}\rangle,$$
$$L^{(1)}\subseteq \langle L_{w_2}, \ldots,
L_{w_{d}}\rangle.$$ The subspace $L^{(1)}\cap L_{w_1}$ is
generated by non-trivial products $[x_i, y_j]$ such that $x_i\in
L_i, \,\,y_j\in L_j$ with $i,j\in \Omega$ and $i+j\equiv w_1\,
(\mathrm{mod}\, n)$. By Lemma \ref{ntl} either $i,j>w_1$ or
$i,j<w_1$. Since there are no non-trivial components $L_i$ with
$i<w_1$, it follows that $i, j\in \{w_2, \ldots,w_{d}\}$.  This
implies (\ref{f-small-comp-1}) and also (\ref{f-small-comp-2}) for
$k=1$.

\vskip1ex Let's now $k>1$. We prove the  inclusion
(\ref{f-small-comp-1}) using the induction hypothesis for the
inclusion (\ref{f-small-comp-2}). The subspace $L^{2^{k-1}}\cap
L_{w_k}$ is generated by the products $[x_i,y_j]$, such that
$x_i\in L_i\cap L^{2^{k-1}-1}$, $y_j\in L_j\cap L^{2^{k-1}-1}$
with $i,j\in \Omega$ and $i+j\equiv w_k\, (\mathrm{mod}\, n)$.  By
induction hypothesis for (\ref{f-small-comp-2}), $y_j\in \langle
L_{w_{k}}, L_{w_{k+1}},\ldots, L_{w_{d}}\rangle$, and hence $y_j$
can be written as a linear combination of products of the form
$$[u_{j_1}, u_{j_2}, \ldots, u_{j_t}]$$  with $j_l\in \{w_k, \ldots, w_d\}$,
$j_1+j_2+\cdots+ j_t\equiv j\,(\mathrm{mod}\, n)$. Applying
repeatedly (\ref{tozh2}), we can represent
$$\big[x_i,[u_{j_1}, u_{j_2}, \ldots, u_{j_t}]\big]$$ as a linear
combination of products
\begin{equation}\label{f-small-comp-3} [x_i, v_{h_1},\ldots, v_{h_t}],\end{equation} where $h_l\in
\{w_k, \ldots, w_d\}$, $h_1+h_2+\cdots +h_t \equiv
j\,(\mathrm{mod}\, n)$. For each such product we have
$$i+h_1+h_2+\cdots +h_t\equiv w_k\,(\mathrm{mod}\, n).$$
If $h_t=w_k$, then $$i+h_1+h_2+\cdots +h_{t-1} \equiv
0\,(\mathrm{mod}\, n)$$ and hence (\ref{f-small-comp-3}) is equal
to zero. If $h_t>w_k$, then $i+h_1+h_2+\cdots +h_{t-1}>w_k$ by
Lemma \ref{ntl}, and consequently the product
(\ref{f-small-comp-3}) lies in $\langle L_{w_{k+1}},
L_{w_{k+1}},\ldots, L_{w_{d}}\rangle$. Note, that the case
$h_t<w_k$ is impossible since all the $h_l$ belong to the set
$\{w_k, \ldots, w_d\}$. As $[x_i,y_j]$ is a linear combination of
products of the form (\ref{f-small-comp-3}) it follows that
$[x_i,y_j]$ also belongs to this subalgebra as required.

\vskip1ex To prove (2) for $k>1$ we apply (2) for $k-1$  to the
subalgebra $L^{2^{k-1}}$:
$$L^{({2^{k}-1)}}=(L^{(2^{k-1})})^{(2^{k-1}-1)}\subseteq \langle (L^{2^{k-1}})\cap
L_{w_k},\dots, (L^{2^{k-1}})\cap L_{w_{d}}\rangle.$$
As we have already proved above, the subspace $(L^{2^{k-1}})\cap
L_{w_k}$ lies in $\langle L_{w_{k+1}},\dots, L_{w_{d}}\rangle .$
Hence
$$L^{(
2^{k}-1)}\subseteq \langle L_{w_{k+1}},\dots, L_{w_{d}}\rangle .$$

\end{proof}

\section{Reduction to graded algebras with ``selective nilpotency'' condition}

Let $L$ be a Lie type algebra that satisfies the hypothesis of
Theorem \ref{th1} and   let  $\varphi$ be a generator and $n$ the
order of the Frobenius kernel $F$. If the ground field $\Bbb K$
contains a primitive $n$th root $\omega$ of 1, we consider the
eigenspaces $L_i=\{x\in L \mid x^{\varphi}=\omega^i x\}$ for the
eigenvalues $\omega^i$. One can verify  that
$$[L_i, L_j]\subseteq L_{i+j\,(\rm{mod}\,n)}\qquad \text{and}\qquad L= \bigoplus _{i=0}^{n-1}L_i,$$
so this is a $(\Bbb Z /n\Bbb Z )$-grading. We also have
$L_0=C_L(F)=0$.

In the case, where $\Bbb K$ does not contain a primitive $n$th
root of 1, by hypothesis of the theorem \ref{th1}, the values
$\alpha, \beta$ in \eqref{tozh} should be constant and not
depending on $a,b,c$.  We extend the ground field by $\omega$ and
denote the resulting Lie algebra by $\widetilde L$. The group $FH$
acts in a natural way on $\widetilde L$ and this action inherits
the conditions that $C_{\widetilde L}(F)=0$ and $C_{\widetilde
L}(H)$ is nilpotent of class~$c$.  Since $\alpha,\beta $ are
constant for all $a,b,c\in L$ in (\ref{tozh}), this property
(\ref{tozh}) holds also in $\widetilde L$, and therefore
$\widetilde L$ is a Lie type algebra. Thus, we can assume that
$L=\widetilde L$ and the ground field contains~$\omega$.

\vskip1ex

A known property of Frobenius groups says that if the Frobenius
kernel is cyclic, then the Frobenius complement  is also cyclic.
Let $h$ be a generator and $q$ the order of $H$ and let
$\varphi^{h^{-1}} = \varphi^{r}$ for some $1\leq r \leq n-1$.
Since by definition of the Frobenius group $C_H(f) = 1$ for every
non-identity $f$ in $F$, it follows that the numbers $n, q, r$
satisfy the following condition

\begin{equation}  \label{prim}
\begin{split}
 & n, q, r \text{ are positive integers such that } 1\leq r \leq n-1 \text{ and } \\
&\quad\text{the image of } r \text{ in } {\Bbb Z}/d{\Bbb Z} \text{
is a primitive } q \text{th root of } 1 \\ &\qquad
\qquad\qquad\text{for every divisor } d \text{ of }n.
\end{split}
\end{equation}

The group $H$ permutes the components $L_i$: $L_i^h=L_{ri}$, since
if $x_i\in L_i$, then $(x_i^h)^{\varphi}=x_i^{h\varphi
h^{-1}h}=(x_i{\varphi^r})^h=\omega^{ir}x_i^h.$

\vskip1ex We can assume that the  characteristic $p$ of the ground
field does not divide $n=|F|$. In the opposite case,  we consider
the Hall $p'$-subgroup $\langle f_1\rangle$ of $F$, the Sylow
$p$-subgroup $\langle f_2\rangle$ of $F$ and $f_2$-invariant
subspace $C_L(f_1)$. If $C_L(f_1)$ is non-trivial, the
automorphism $f_2$ (as a $p$-automorphism acting on a space over a
field of characteristic $p$) has necessary a non-trivial fixed
point on $C_L(f_1)$ which would be also a non-trivial fixed point
for $F$ that contradicts our assumption. Thus $C_L(f_1)=0$ and we
can replace $F$ by $\langle f_1\rangle$, consider the Frobenius
group $\langle f_1\rangle H$ instead of $FH$ and, consequently,
assume that $p$ does not divide~$n$.

\vskip1ex In what follows, to simplify the notations, (under the
Index Convention) we will denote $(x_s)^{h^i}$ by $x_{r^is}$ for
$x_s\in L_s$ . Let $x_{a_1},\dots,x_{a_{c+1}}$ be homogeneous
elements in $L_{a_1},\dots,L_{a_{c+1}}$, respectively. Consider
the sums
\begin{align*}
X_1&=x_{a_1}+x_{ra_1}+\cdots+x_{r^{q-1}a_1},\\
\vdots&\\
X_{c+1}&=x_{a_{c+1}}+x_{ra_{c+1}}+\cdots+x_{r^{q-1}a_{c+1}}.
\end{align*}
Since all of them lie in subalgebra $C_L(H)$, which is nilpotent
of class $c$, it follows that
$$[X_1,\ldots, X_{c+1}]=0.$$
We expand the expressions to obtain on the left a linear
combination of products in the $x_{r^ja_i}$, which in particular
involves the term $[x_{a_1},\ldots, x_{a_{c+1}}]$. Suppose that
the product $[x_{a_1},\ldots, x_{a_{c+1}}]$ is non-zero. Then
there must be other terms in the expanded expression that belong
to the same component $L_{a_1+\cdots+a_{c+1}}$. In other words,
then
$$a_{1}+\dots+a_{c+1}=r^{\alpha_1}a_{1}+\dots+r^{\alpha_{c+1}}a_{c+1}$$
for some $\alpha_i\in\{0,1,2,\dots,q-1\}$ not all of which are
zeros.  Equivalently, if for all $(\alpha_1,
\alpha_2,\ldots,\alpha_{c+1})\neq (0,0,\ldots 0)$ with
$\alpha_i\in\{0,1,2,\dots,q-1\}$,
$$a_{1}+\dots+a_{c+1}\neq
r^{\alpha_1}a_{1}+\dots+r^{\alpha_{c+1}}a_{c+1},$$ then the
product $[x_{a_1},\ldots x_{a_{c+1}}]$ is equal to zero.

\vskip1ex  The above considerations lead to the following notion
that plays an important role in further arguments.

\begin{definition}
Let $a_1,\dots,a_k$ be not necessarily distinct non-zero elements
of $\Bbb Z/n\Bbb Z$. We say that the sequence $(a_1,\dots,a_k)$ is
\textit{$r$-dependent} if
$$a_{1}+\dots+a_{k}=r^{\alpha_1}a_{1}+\dots+r^{\alpha_k}a_{k}$$
for some $\alpha_i \in\{0,1,2,\dots,q-1\}$ not all of which are
zero. If the sequence $(a_1,\dots,a_k)$ is not $r$-dependent, i.e.
if for all $(\alpha_1, \alpha_2,\ldots,\alpha_{k})\neq (0,0,\ldots
0)$ with $\alpha_i\in\{0,1,2,\dots,q-1\}$,
$$a_{1}+\dots+a_{k}\neq
r^{\alpha_1}a_{1}+\dots+r^{\alpha_{k}}a_{k},$$ we call it
\textit{$r$-independent}.
\end{definition}

\begin{remark}\label{remark1}
A single non-zero element $a\in \Bbb Z/n\Bbb Z$ is always
$r$-independent. Indeed,  if $a=r^{\alpha}a$ for $\alpha
\in\{1,2,\dots,q-1\}$, then $a=0$ by \eqref{prim}.
\end{remark}

\begin{definition}
 Let $n, q, r$ be integers defined by \eqref{prim}. We say that a $(\Z/n\Z)$-graded Lie type algebra $L$ satisfies
the \textit{selective $c$-nilpotency condition} if, under the
Index Convention,
\begin{equation}\label{select}
[x_{d_1},x_{d_2},\dots, x_{d_{c+1}}]=0\quad \text{ whenever }
(d_1,\dots,d_{c+1}) \text{ is $r$-independent}.
\end{equation}
 \end{definition}

\begin{remark}\label{remark2}
If $c=0$  a $(\Z/n\Z)$-graded Lie type algebra $L$ satisfies the
\textit{selective $c$-nilpotency condition} if and only if $L_d=0$
for all $d\ne0$, since any element $d\ne 0$ is $r$-independent by
Remark~\ref{remark1}.

\end{remark}

Summarizing all of the above we can assert that to demonstrate
Theorem \ref{th1}  it suffices to prove the nilpotency of
$(q,c)$-bounded class of a $(\Z/n\Z)$-graded Lie type algebra with
selective $c$-nilpotency condition.

\section{ Bounding of derived length }

In this section we suppose that $L$ is  a $(\Z/n\Z)$-graded Lie
type algebra with $L_0=0$ that satisfies   the selective
$c$-nilpotency condition (\ref{select}).  Note, that from
Proposition \ref{small-comp} it follows that $L$ is already
solvable of $n$-bounded derived length. We will obtain a bound of
the derived length that does not depend of $n$, but depends only
on $q=|H|$ and $c$.

\vskip1ex

We start with  an elementary fact from \cite{khu-ma-shu}  on
$r$-dependent sequences.

\begin{notation}
For a given $r$-independent sequence $(a_1,\dots,a_k) $ we denote
by $D(a_1,\dots,a_k)$ the set of all $j\in\Bbb Z/n\Bbb Z$ such
that $(a_1,\dots,a_k,j)$ is $r$-dependent.
\end{notation}

\begin{lemma} [{\cite[Lemma 4.4]{khu-ma-shu}}]\label{115}
If $(a_1,\dots,a_k)$ is $r$-independent, then
$$|D(a_1,\dots,a_k)|\leq q^{k+1}.$$
\end{lemma}

\vskip1ex
 We now present a series of lemmas that were proved in \cite{khu-ma-shu} for Lie algebras but need only some minor modifications to be used for the case of Lie type algebras.

\vskip1ex

\begin{notation} The order of an element $b\in\Bbb Z/n\Bbb Z$ (in
the additive group) is denoted by $o(b)$.
\end{notation}

\begin{lemma}[{\cite[Lemma 4.6]{khu-ma-shu}}]\label{lb}
Suppose that a  $(\Z/n\Z)$-graded Lie type algebra $L$ with
$L_0=0$ satisfies the selective $c$-nilpotency
condition~\eqref{select}. Let   $b$ be an element of
$\Bbb{Z}/n\Bbb{Z}$ such that $o(b)>2^{2^{2q-3}-1}c^{2^{2q-3}}$.
Then there are at most $q^{c+1}$ elements $a\in\Bbb{Z}/n\Bbb{Z}$
such that $[L_a, \underbrace{L_b,\dots,L_b}_{c}] \neq 0$.
\end{lemma}

\begin{proof} The proof is exactly the same as that of
Lemma~4.6  in \cite{khu-ma-shu}.
\end{proof}

\begin{lemma}[{\cite[Lemma 4.7]{khu-ma-shu}}]\label{l_b}
Suppose that a $(\Z/n\Z)$-graded Lie type algebra $L$ with $L_0=0$
satisfies the selective $c$-nilpotency condition~\eqref{select}.
There is a $(c,q)$-bounded number $w$ such that
$$[
L,\underbrace{L_b,\dots,L_b}_{w}] =0$$ whenever $b$ is an element
of $\Bbb{Z}/n\Bbb{Z}$ such that
$o(b)>\max\{2^{2^{2q-3}-1}c^{2^{2q-3}},q^{c+1} \}$.
\end{lemma}
\begin{proof}  See Lemma 4.7 in \cite{khu-ma-shu}.
\end{proof}

\begin{lemma}[{\cite[Lemma 4.8]{khu-ma-shu}}]\label{odin} Let $(d_1,\dots, d_c)$ be an arbitrary $r$-independent
sequence and $U=[u_{d_1},\dots,u_{d_c}]$ be a homogeneous product
with indices $(d_1,\dots, d_c)$ (under Index Convention). Suppose
that a $(\Z/n\Z)$-graded Lie type algebra $L$ with $L_0=0$
satisfies the selective $c$-nilpotency condition~\eqref{select}.
Then
 \begin{enumerate}[a)]
  \item
 every product of the form
\begin{equation}\label{eq4-1}
[U,x_{i_1},\dots,x_{i_t}]
\end{equation}
 can be
written as a linear combination of products of the form
\begin{equation}\label{eq5-1}
[U, m_{j_1},\dots,m_{j_{s}}],
\end{equation}
where $j_k\in D(d_1,\dots,d_c)$ and $s\leq t$. The case $s=t$ is
possible only if $i_k\in D(d_1,\dots, d_c)$ for all $k=1,\dots,t$.
\item every product of the form
\begin{equation}\label{eq4}
[x_{i_1}, U, \dots,x_{i_t}]
\end{equation}
 can be
written as a linear combination of products of the form
\begin{equation}\label{eq5}
[m_{j_1},U,\dots,m_{j_{s}}],
\end{equation}
where $j_k\in D(d_1,\dots,d_c)$ and $s\leq t$. The case $s=t$ is
possible only if $i_k\in D(d_1,\dots, d_c)$ for all $k=1,\dots,t$.
\end{enumerate}
\end{lemma}

\begin{proof} The part a) is proved in the same way as Lemma 4.8 in
\cite{khu-ma-shu} applying everywhere the property~(\ref{tozh})
instead of the  Jacobi identity.

\vskip1ex  To prove part b) we use induction on $t$.  If $t=0$,
there is nothing to prove. If $t=1$ and $i_1\in D(d_1,\dots,d_c)$,
then $[x_{i_1}, U]$ is of the required form. If $i_1\notin
D(d_1,\dots,d_c)$, then $[x_{i_1}, U]=0$ by \eqref{select}.

\vskip1ex

 Let  $t>1$. If all the indices $i_j$
belong to $D(d_1,\dots,d_c)$, then the product
$[x_{i_1},U,\dots,x_{i_t}] $ is of the required form with $s=t$.
Suppose that in \eqref{eq4} there is an element $x_{i_k}$ with the
index $i_k$ that does not belong to $D(d_1,\dots,d_c)$. Let $k$ be
as small as possible.  We use $k$ as a second induction parameter.

\vskip1ex If $k=1$, then the product \eqref{eq4} is zero by
\eqref{select} and we are done. If $k=2$ we rewrite product
\eqref{eq4} using (\ref{tozh}) as
$$[x_{i_{1}}, U, x_{i_{2}}, \ldots, x_{i_t}]=\alpha [x_{i_{1}}, [U,x_{i_2}], \ldots, x_{i_t}]+\beta [[x_{i_{1}},
x_{i_2}],U,\ldots, x_{i_t}]$$$$= \alpha [x_{i_{1}}, [U,
x_{i_2}],\ldots, x_{i_t}]+\beta [[x_{i_{1}+i_2},U,\ldots,
x_{i_t}],
$$
where $x_{i_{1}+i_2}=[x_{i_{1}}, x_{i_2}]$ (under the Index
Condition). The first term is trivial by \eqref{select} because
$x_{i_2}\notin D(d_1,\dots,d_c)$.  The second term is of required
form by induction hypothesis because it is shorter than the
original one.

\vskip1ex Suppose that $k\geq 3$. We  rewrite \eqref{eq4} using
\eqref{tozh} as
$$[x_{i_{1}},U,\dots,x_{i_{k-1}},x_{i_k},\dots,x_{i_t}]=\alpha \big[x_{i_{1}},U,\dots,[x_{i_{k-1}},x_{i_k}],
\dots,x_{i_t}\big]+$$
$$+ \beta [x_{i_{1}},U,\dots,x_{i_k},x_{i_{k-1}},\dots,x_{i_t}].$$
By the induction hypothesis the first term is a linear combination
of products of the form~\eqref{eq5} because it is shorter than
\eqref{eq4}, while the second term has the required form because
the index that does not belong to $D(d_1,\dots,d_c)$ here occurs
closer to $U$ than in \eqref{eq4}.
\end{proof}

Let $D=|D(d_1,\dots,d_c)|$ and let $w$ be the number given by
Lemma~\ref{l_b}.

\begin{lemma}[{\cite[Lemma 4.10]{khu-ma-shu}}]\label{dva}
Let $D=|D(d_1,\dots,d_c)|$ and let $w$ be the number given by
Lemma~\ref{l_b}. Suppose further that $L$ and $U$ are as in
Lemma~\ref{odin}. Then the ideal of $L$ generated by $U$
 is spanned by
products of the form
\begin{equation}\label{eq6-1}
[U, m_{i_1},\dots,m_{i_u},m_{i_{u+1}},\dots,m_{i_{v}}]
\end{equation}
and
\begin{equation}\label{eq6-2}
[m_{i_1},U, m_{i_2},\dots,m_{i_u},m_{i_{u+1}},\dots,m_{i_{v}}]
\end{equation}
 where $u\leq (w-1)D+1$, $i_k\in D(d_1,\dots, d_c)$ for all $k=1,2\ldots v$, and  $o(i_k)\leq N(c,q)$ for $k>u$.
\end{lemma}

\begin{proof} By Lemma \ref{odin} the ideal  generated by $U$
is spanned by the products of the two forms \eqref{eq5-1} and
\eqref{eq5}. We denote this span by $R$. Exactly in the same
manner as in Lemma 4.10 in \cite{khu-ma-shu}), we prove by
induction on the length of the products  that \eqref{eq5-1} and
\eqref{eq5} do not change modulo $R$ under any permutation of the
$m_{j_k}$, $k\geq 3$. To adapt the proof in \cite{khu-ma-shu} we
only need to replace the Jacobi identity by \eqref{tozh} and Lemma
4.8 in~\cite{khu-ma-shu} by Lemma~\ref{odin}.

If among the $m_{j_k}$ there are at least $w$ elements with the
same index $j_k$ such that $o(j_k)\geq N(c,q)$, we move these
elements next to each other. Then by Lemma~\ref{l_b} the
corresponding product is equal to zero. If all the indices $j_k\in
A$ occur less than $w$ times, we place all these elements right
after the $[U,m_{j_1}]$ or respectively $[m_{j_1},U]$. This
initial segment has length at most $D(w-1)+c+1$, so the resulting
products take the required form~\eqref{eq6-1} or \eqref{eq6-2}.
\end{proof}

\begin{proposition}[{\cite[Corollary 4.11]{khu-ma-shu}}]\label{malocomp}
Suppose that a $(\Z/n\Z)$-graded Lie type algebra  $L$ with
$L_0=0$ satisfies the selective $c$-nilpotency
condition~\eqref{select}, and let $(d_1,\dots, d_c)$ be an
$r$-independent sequence. Then the ideal $_{\rm id}\langle
[L_{d_1},\dots, L_{d_c}]\rangle$ has $(c,q)$-boundedly many
non-trivial components of the induced grading.
\end{proposition}

\begin{proof} The proof can be easily reconstructed from the proof of Corollary 4.11 in
\cite{khu-ma-shu}.  We only need  to replace Lemmas  4.4, 4.7, 4.8
and  4.10 in \cite{khu-ma-shu} by  Lemmas \ref{115}, \ref{l_b},
\ref{odin} and \ref{dva} respectively.
\end{proof}

\begin{lemma}[{\cite[Lemma 4.12]{khu-ma-shu}}]\label{pyat}
Suppose that a homogeneous ideal $T$ of a Lie type algebra $L$ has
only $e$ non-trivial components. Then $L$ has at most $e^2$
components that do not centralize $T$.
\end{lemma}

\begin{proposition}[{\cite[Proposition 4.13]{khu-ma-shu}}]\label{razresh}
Suppose that a $(\Z/n\Z)$-graded Lie type algebra $L$ with $L_0=0$
satisfies the selective $c$-nilpotency condition~\eqref{select}.
Then $L$ is solvable of $(c,q)$-bounded derived length $f(c,q)$.
\end{proposition}

\begin{proof} We reproduce the proof of Corollary 5.10 in~\cite{khu-ma-shu} replacing the Jacobi identity by the property
\eqref{tozh} and applying Propostion \ref{small-comp} instead of
Shalev--Kreknin theorem \cite{kr, shalev}.

\vskip1ex
 We use induction on
$c$. If $c=0$, then $L=0$ by Remark~\ref{remark2} and we are done.

\vskip1ex Let $c\geq 1$.  We consider the ideal $I$  of $L$
generated by all products $[L_{i_1},\dots,L_{i_c}]$, where
$(i_1,\dots,i_c)$ ranges through all $r$-independent sequences of
length~$c$. The quotient algebra $L/I$ has  induced
$(\Z/n\Z)$-grading of $L/I$ with trivial zero-component and $L/I$
satisfies the selective $(c-1)$-nilpotency condition. It follows
by the induction hypothesis that $L/I$ is solvable of bounded
derived length, say, $f_0$, that is, $L^{(f_0)}\leq I$.

\vskip1ex Let now $(i_1,\dots,i_c)$ be  a $r$-independent sequence
$(i_1,\dots,i_c)$.  We set
$$T={}_{\rm id}\langle[L_{i_1},\dots,L_{i_c}]\rangle.$$ Proposition \ref{malocomp} implies that there are only $(c,q)$-boundedly many, say,
$e$, non-trivial grading components in $T$. By Lemma \ref{pyat}
there are at most $e^2$ components that do not centralize $T$. The
subalgebra $C_L(T)$ is also a homogeneous ideal by \eqref{tozh},
since
$$[C_L(T), L, T]\subseteq \big[C_L(T), [L, T]\big]+\big[C_L(T), T, L\big]\subseteq
[C_L(T), T]\subseteq \{0\},$$

$$[L, C_L(T),  T]\subseteq \big[L, [C_L(T),  T]\big]+\big[L, T, C_L(T)\big]\subseteq
[T, C_L(T)]\subseteq \{0\}.$$ The quotient algebra $L/C_L(T)$ has
induced $(\Z/n\Z)$-grading with trivial zero-component and with at
most $e^2$ non-trivial components.  By Proposition
\ref{small-comp} the algebra $L/C_L(T)$ is solvable of $e$-bounded
derived length, say, $f_1$. Therefore $L^{(f_1)}\subseteq C_L(T)$
and $[L^{(f_1)}, T]=0$.  Since $f_1$ does not depend on the choice
of the $r$-independent tuple $(i_1,\dots,i_c)$ and $I$ is the sum
of all such ideals $T$, it follows that $[L^{(f_1)},I]=0$. Recall
that $L^{(f_0)}\leq I$. Hence, $[L^{(f_1)},L^{(f_0)}]=0$. Thus $L$
is solvable of $(c,q)$-bounded derived length at most
$\max\{f_0,f_1\}+1$.
\end{proof}

\section{ Bounding of nilpotency class }

In this section we complete the proof of Theorem \ref{th1} by
proving the nilpotency of $(c,q)$-bounded class of a
$(\Z/n\Z)$-graded Lie type $L$ algebra with $L_0=0$ that satisfies
the selective $c$-nilpotency condition (\ref{select}). We have
already proved that $L$ is solvable of $(c,q)$-bounded derived
length and can use induction on the derived length of $L$.

\vskip1ex
 If $L$ is abelian,
there is nothing to prove. Assume that $L$ is metabelian, that is
$[[L,L], [L,L]]=0$.

\vskip1ex We will use two following lemmas from \cite{khu-ma-shu}.

\begin{lemma}[{\cite[Lemma 5.2]{khu-ma-shu}}]\label{l_b-metab} Let $L$ be a metabelian
$(\Z/n\Z)$-graded Lie type $L$ algebra with $L_0=0$ that satisfies
the selective $c$-nilpotency condition (\ref{select}). Then there
is a $(c,q)$-bounded number $m$ such that
{$[L,\underbrace{L_b,\dots,L_b}_{m} ]=0$} for every $b\in
\Bbb{Z}/n\Bbb{Z}$.
\end{lemma}
\begin{proof} See Lemma 5.2 in \cite{khu-ma-shu}.
\end{proof}

\begin{lemma}[{\cite[Lemma 4.5]{khu-ma-shu}}]\label{rigid}
Suppose that for some $m$ a sequence $(a_1,\dots,a_k)$ of non-zero
elements of $\Bbb Z/n\Bbb Z$ contains at least $q^m+m$ different
values. Then one can choose an $r$-independent subsequence
$(a_1,a_{i_2},\dots,a_{i_m})$ of $m$ elements that contains $a_1$.
\end{lemma}
\begin{proof} See Lemma 4.5 in \cite{khu-ma-shu}.
\end{proof}

 Let $m=m(c,q)$ be as in
Lemma~\ref{l_b-metab} and put $g=(m-1)(q^{c+1}+c)+2$. We consider
the product $[[L,L]_{a_1},L_{a_2},\dots,L_{a_g}]$, where
$a_1,\dots,a_g\in  \Bbb Z/n\Bbb Z$ are non-zero. If the sequence
$(a_1,\dots,a_g)$ contains an $r$-independent sequence of length
$c+1$ that starts with $a_1$, we permute  the $L_{a_i}$ in order
to have an initial segment  with indices which form an
$r$-independent subsequence $a_1,\dots,a_{c+1}$. Then
$[[L,L]_{a_1},L_{a_2},\dots,L_{a_g}]=0$ by~\eqref{select}. If the
sequence $(a_1,\dots,a_g)$ does not contain an $r$-independent
subsequence of length $c+1$ starting with $a_1$, then by
Lemma~\ref{rigid} the sequence $(a_1,\dots,a_g)$ contains at most
$q^{c+1}+c$ different values. It follows that either the value of
$a_1$ occurs in $(a_1,\dots,a_g)$ at least $m+1$ times or, else,
another value, different from $a_1$, occurs at least $m$ times. In
any case there are $m$ components $L_{a_i}$ with the same index,
say  $j$. We move these  components $L_j$ next to $L_{a_1}$. It
follows from Lemma~\ref{l_b-metab} that
$[[L,L]_{a_1},L_{a_2},\dots,L_{a_g}]=0$. Thus, we conclude that
$L$ is nilpotent of class at most $g$.

\vskip1ex Now suppose that the derived length of $L$ is at
least~3. By the induction hypothesis, $[L,L]$ is nilpotent of
bounded class. The quotient algebra $L/[[L,L],[L,L]]$ is
metabelian and hence  nilpotent of bounded class. It follows that
$L$ is nilpotent of bounded  nilpotency class  by the analogue of
P.~Hall's theorem (Lemma \ref{chao}).

\end{document}